\documentclass[10pt,a4paper]{article}
\usepackage[a4paper]{geometry}
\usepackage{amssymb,latexsym,amsmath,amsfonts,amsthm}
\usepackage{graphicx}

\usepackage{epsfig}

\newcommand{\supp}{\mathrm{supp}}

\newtheorem{theorem}{Theorem}[section]
\newtheorem{lemma}[theorem]{Lemma}

\newtheorem{corollary}[theorem]{Corollary}

\theoremstyle{definition}
\newtheorem{definition}[theorem]{Definition}

\theoremstyle{remark}

\newtheorem{remark}[theorem]{Remark}

\numberwithin{equation}{section}

\title{On the convergence of type I Hermite-Pad\'e approximants}

\date{\today}

\author{G. L\'{o}pez Lagomasino\footnotemark[2], S. Medina Peralta\footnotemark[2]}

\begin{document}

\maketitle
\renewcommand{\thefootnote}{\fnsymbol{footnote}}
\footnotetext[2]{Departamento de
Matem\'{a}ticas, Universidad Carlos III de Madrid, Avda. Universidad
30, 28911 Legan\'{e}s, Madrid, Spain. email: lago\symbol{'100}math.uc3m.es.\\
Both authors were partially supported
by research grant MTM2012-36372-C03-01 of Ministerio de Econom\'{\i}a y Competitividad, Spain.}

\begin{abstract}
Pad\'e approximation has two natural extensions to vector rational approximation through the so called type I and type II Hermite-Pad\'e approximants. The convergence properties of type II Hermite-Pad\'e approximants have been studied.  For such approximants Markov and Stieltjes type theorems are available. To the present,  such results have not been obtained for type I approximants. In this paper, we provide Markov and Stieltjes type theorems on the convergence of type I Hermite-Pad\'e approximants for Nikishin systems of functions.
\end{abstract}



\maketitle

\section{Introduction}
\label{section:intro}

Let $s$ be a finite  Borel measure with constant (not neessarily positive) sign whose support $\mbox{supp}(s)$ contains infinitely many points and is contained in the real line $\mathbb{R}$. If $\mbox{supp}(s)$ is an unbounded set we assume additionally that   $x^n \in L_1(s), n \in \mathbb{N}$.  By $\Delta = \mbox{Co}(\mbox{supp}(s))$ we denote the smallest interval which contains  $\supp(s)$.   We denote this class of measures by ${\mathcal{M}}(\Delta)$. Let
\[ \widehat{s}(z) = \int\frac{d s(x)}{z-x}
\]
be the Cauchy transform of $s$.

Given any positive integer $n \in \mathbb{N}$ there exist polynomials $Q_n,P_n$
satisfying:
\begin{itemize}
\item $\deg Q_n \leq n,\quad  \deg P_n \leq n-1, \quad Q_n \not\equiv 0,$
\item $(Q_n \widehat{s} -P_n)(z) = \mathcal{O}(1/z^{n+1}),\quad  z \to \infty.$
\end{itemize}
The ratio  $\pi_n = P_n/Q_n$ of any two such polynomials defines a unique rational function  called the $n$th term of the diagonal sequence of Pad\'e approximants to $\widehat{s}$. Cauchy transforms of measures are important: for example, many elementary functions may be expressed through them, the resolvent function of a bounded selfadjoint operator adopts that form, and they characterize all functions holomorphic in the upper half plane whose image lies in the lower half plane and can be extended continuously to the complement of a finite segment $[a,b]$ of the real line taking negative values for $z < a$ and positive values for $z > b$ (then $\mbox{supp}(s) \subset [a,b]$), see \cite[Theorem A.6]{KN}. Providing efficient methods for their approximation is a central question in the theory of rational approximation.

When $\Delta$ is bounded,  A.A. Markov proved in \cite{Mar} (in the context of the theory of continued fractions) that
\begin{equation} \label{Mar} \lim_{n \to \infty} \pi_n(z) = \widehat{s}(z)
\end{equation}
uniformly on each compact subset of $\overline{\mathbb{C}} \setminus \Delta$.  It is easy to deduce that this limit takes place with geometric rate. When $\Delta$ is a half line,  T.J. Stieltjes in  \cite{Sti} showed that \eqref{Mar} takes place if and only if the moment problem for the sequence $\left(c_n\right)_{n\geq 0}, c_n = \int x^n ds(x),$ is determinate. It is well known that the moment problem for measures of bounded support  is   determinate; therefore, Stieltjes' theorem contains Markov's result. In \cite{Car}, T. Carleman  proved  when $\Delta \subset \mathbb{R}_+$ that
\begin{equation} \label{Carle} \sum_{n \geq 1} |c_{n}|^{-1/2n} = \infty
\end{equation}
is sufficient for the moment problem to be determinate. For an arbitrary measure  $s \in \mathcal{M}(\Delta)$, where $\Delta$ is contained in a half line, we say that it satisfies Carleman's condition if after an affine transformation which takes $\Delta$ into $\mathbb{R}_+$ the image measure satisfies Carleman's condition.

Pad\'e  approximation has two natural extensions to the case of vector rational approximation.  These extensions  were introduced by Hermite in order to study the transcendency of $e$. Other applications in number theory have been obtained. See \cite{Ass} for a survey of results in this direction. Recently, these approximants and their associated Hermite-Pad\'e polynomials have appeared in a natural way in certain models coming from probability theory and mathematical physics. A summary of this type of applications can be found in \cite{Kuij}

Given a system of finite Borel measures $S = (s_1,\ldots,s_m)$ with constant sign and a multi-index ${\bf n} = (n_1,\ldots,n_m) \in \mathbb{Z}_+^m \setminus \{{\bf 0}\}, |{\bf n}|= n_1+\cdots+n_m$, where $\mathbb{Z}_+$ denotes the set of non-negative integers and $\bf 0$ the $m$-dimensional zero vector, their exist polynomials $a_{{\bf n},j}, j=0,\ldots,m,$ not all identically equal to zero, such that:
\begin{itemize}
\item $\deg a_{{\bf n},j} \leq n_j -1, j=1,\ldots,m, \quad \deg a_{{\bf n},0} \leq \max(n_j) -2,$
\item $a_{{\bf n},0}(z) + \sum_{j=1}^m a_{{\bf n},j}(z) \widehat{s}_j(z) = \mathcal{O}(1/z^{|{\bf n}|}),\,\,\, z \to \infty \,\,\,(\deg a_{{\bf n},j} \leq -1$ means that $a_{{\bf n},j} \equiv 0$).
\end{itemize}
Analogously, there exist polynomials $Q_{\bf n}, P_{{\bf n},j}, j=1,\ldots,m$, satisfying:
\begin{itemize}
\item $\deg Q_{\bf n} \leq |{\bf n}|, Q_{\bf n} \not\equiv 0, \quad \deg P_{{\bf n},j} \leq |{\bf n}|-1, j=1,\ldots,m,$
\item $Q_{\bf n}(z) \widehat{s}_j (z) - P_{{\bf n},j}(z) = \mathcal{O}(1/z^{n_j +1}), \quad z\to \infty, \quad j=1,\ldots,m.$
\end{itemize}
Traditionally, the systems of polynomials $(a_{{\bf n},0}, \ldots,a_{{\bf n},m})$ and $(Q_{\bf n},P_{{\bf n},1},\ldots,P_{{\bf n},m})$ have been called type I and type II Hermite-Pad\'e approximants (polynomials) of $(\widehat{s}_1,\ldots,\widehat{s}_m)$, respectively. When $m=1$ both definitions reduce to that of classical Pad\'e approximation.

From the definition, type II Hermite-Pad\'e approximation is easy to view as an approximating scheme of the vector function
$(\widehat{s}_1,\ldots,\widehat{s}_m)$ by considering a sequence of vector rational functions of the form $(P_{{\bf n},1}/Q_{\bf n},\ldots,P_{{\bf n},m}/Q_{\bf n}), {\bf n} \in \Lambda \subset \mathbb{Z}_+^m$, where $Q_{\bf n}$ is a common denominator for all components. Regarding type I, it is not obvious what is the object to be approximated or even what should be considered as the approximant. Our goal is to clarify these questions providing straightforward analogues of the Markov and Stieltjes theorems.

Before stating our main result, let us introduce what is called a Nikishin system of measures to which we will restrict our study.
Let $\Delta_{\alpha}, \Delta_{\beta}$ be two intervals contained in the real line which have at most one point in common, $\sigma_{\alpha} \in {\mathcal{M}}(\Delta_{\alpha}), \sigma_{\beta} \in {\mathcal{M}}(\Delta_{\beta})$, and $\widehat{\sigma}_{\beta} \in L_1(\sigma_{\alpha})$. With these two measures we define a third one as follows (using the differential notation)
\[ d \langle \sigma_{\alpha},\sigma_{\beta} \rangle (x) := \widehat{\sigma}_{\beta}(x) d\sigma_{\alpha}(x).
\]
Above, $\widehat{\sigma}_{\beta}$ denotes the Cauchy transform of the measure $\sigma_{\beta}$.  The more appropriate notation $\widehat{\sigma_{\beta}}$ causes space consumption and aesthetic inconveniences. We need to take consecutive products of measures; for example,
\[\langle \sigma_{\gamma},  \sigma_{\alpha},\sigma_{\beta} \rangle :=\langle \sigma_{\gamma}, \langle \sigma_{\alpha},\sigma_{\beta} \rangle \rangle. \]
 Here, we assume not only that $\widehat{\sigma}_{\beta} \in L_1(\sigma_{\alpha})$ but also $\langle \sigma_{\alpha},\sigma_{\beta} \widehat{\rangle} \in L_1(\sigma_{\gamma})$ where $\langle \sigma_{\alpha},\sigma_{\beta} \widehat{\rangle}$ denotes the Cauchy transform of $\langle \sigma_{\alpha},\sigma_{\beta}  {\rangle}$. Inductively, one defines products of a finite number of measures.

\begin{definition} \label{Nikishin} Take a collection  $\Delta_j, j=1,\ldots,m,$ of intervals such that, for   $j=1,\ldots,m-1$
\[ \Delta_j \cap \Delta_{j+1} = \emptyset, \qquad \mbox{or} \qquad \Delta_j \cap \Delta_{j+1} = \{x_{j,j+1}\},
\]
where $x_{j,j+1}$ is a single point. Let $(\sigma_1,\ldots,\sigma_m)$ be a system of measures such that $\mbox{Co}(\supp (\sigma_j)) = \Delta_j, \sigma_j \in {\mathcal{M}}(\Delta_j), j=1,\ldots,m,$  and
\begin{equation} \label{eq:autom}
\langle \sigma_{j},\ldots,\sigma_k  {\rangle} := \langle \sigma_j,\langle \sigma_{j+1},\ldots,\sigma_k\rangle\rangle\in {\mathcal{M}}(\Delta_j),  \qquad  1 \leq j < k\leq m.
\end{equation}
When $\Delta_j \cap \Delta_{j+1} = \{x_{j,j+1}\}$ we also assume that $x_{j,j+1}$ is not a mass point of either $\sigma_j$ or $\sigma_{j+1}$.
We say that $(s_1,\ldots,s_m) = {\mathcal{N}}(\sigma_1,\ldots,\sigma_m)$, where
\[ s_1 = \sigma_1, \quad s_2 = \langle \sigma_1,\sigma_2 \rangle, \ldots \quad , s_m = \langle \sigma_1, \sigma_2,\ldots,\sigma_m  \rangle
\]
is the Nikishin system of measures generated by $(\sigma_1,\ldots,\sigma_m)$.
\end{definition}

Initially, E.M. Nikishin in \cite{Nik}  restricted himself to measures with bounded support and no intersection points between consecutive $\Delta_j$.
Definition \ref{Nikishin} includes interesting examples  which appear in practice (see, \cite[Subsection 1.4]{FL4}). We follow the approach of \cite[Definition 1.2]{FL4}  assuming additionally the existence of all the moments of the generating measures. This is done only for the purpose of simplifying the presentation without affecting too much the generality. However, we wish to point out that the results of this paper have appropriate formulations with the definition given in \cite{FL4} of a Nikishin system.

When $m=2$, for multi-indices of the form $ {\bf n } = (n,n)$ E.M. Nikishin proved  in \cite{Nik}  that
\[ \lim_{n \to \infty} \frac{P_{{\bf n},j}(z)}{Q_{\bf n}(z)} = \widehat{s}_j(z), \qquad j=1,2,
\]
uniformly on each compact subset of $\overline{\mathbb{C}} \setminus \Delta_1$
In \cite{BL} this result was extended to any Nikishin system of $m$ measures including generating measures with unbounded support. The convergence for more general  sequences of multi-indices was treated in \cite{FL1}, \cite{FL2} and \cite{GRS}.

In \cite[Lemma 2.9]{FL4} it was shown that if $(\sigma_1,\ldots,\sigma_m)$ is a generator of a Nikishin system then $(\sigma_m,\ldots,\sigma_1)$ is also a generator (as well as any subsystem of consecutive measures drawn from them). When the supports are bounded and consecutive supports do not intersect this is trivially true. In the following, for $1\leq j\leq k\leq m$ we denote
\[ s_{j,k} := \langle \sigma_j,\sigma_{j+1},\ldots,\sigma_k \rangle, \qquad s_{k,j} := \langle \sigma_k,\sigma_{k-1},\ldots,\sigma_j \rangle.
\]

To state our main results, the natural framework is that of multi-point type I Hermite-Pad\'e approximation.

\begin{definition} \label{type1} Let $(s_{1,1},\ldots,s_{1,m}) = \mathcal{N}(\sigma_1,\ldots,\sigma_m), {\bf n} = (n_1,\ldots,n_m) \in \mathbb{Z}_+^{m} \setminus \{{\bf 0}\},$ and $w_{\bf n},$ $ \deg w_{\bf n} \leq |{\bf n}| + \max(n_j)-2,$ a polynomial with real coefficients whose zeros lie in $\mathbb{C} \setminus\Delta_1$, be given. We say that $\left(a_{{\bf n},0}, \ldots, a_{{\bf n},m}\right)$ is a  type I multi-point Hermite-Pad\'e approximation of $(\widehat{s}_{1,1},\ldots,\widehat{s}_{1,m})$ with respect to $w_{\bf n}$ if:
\begin{itemize}
\item[i)] $\deg a_{{\bf n},j} \leq n_j -1, j=1,\ldots,m, \quad \deg a_{{\bf n},0} \leq n_0 -1, \quad n_0 := \max_{j=1,\ldots,m} (n_j) -1,\quad$ not all identically equal to $0$ ($n_j = 0$ implies that $a_{{\bf n},j} \equiv 0$),
\item[ii)] $\mathcal{A}_{{\bf n},0}/w_{\bf n} \in \mathcal{H}(\mathbb{C} \setminus \Delta_1)\quad$ and  $\quad \mathcal{A}_{{\bf n},0}(z)/w_{\bf n}(z)= \mathcal{O}(1/z^{|{\bf n}|}), \quad z\to \infty, \quad$ where
    \[ \mathcal{A}_{{\bf n},j}  := a_{{\bf n},j}  + \sum_{k= j+1}^m a_{{\bf n},k}  \widehat{s}_{j+1,k}(z), \quad j=0,\ldots,m-1, \qquad \mathcal{A}_{{\bf n},m}  :=  {a}_{{\bf n},m}.
    \]
\end{itemize}
\end{definition}

If $\deg w_{\bf n} = |{\bf n}| + \max(n_j)-2$ the second part of ii) is automatically fulfilled. Should $\deg w_{\bf n} = N  < |{\bf n}| + \max(n_j)-2$ then $|{\bf n}| + \max(n_j)-2 -N$ (asymptotic) interpolation conditions are imposed at $\infty$. In general $|{\bf n}| + \max(n_j)-2$ interpolation conditions are imposed at points in $({\mathbb{C}}\setminus \Delta_1) \cup \{\infty\}$. The total number of free parameters (the coefficients of the polynomials $a_{{\bf n},j}, j=0.\ldots,m$) equals $|{\bf n}| + \max(n_j)-1$; therefore, the homogeneous linear system of equations to be solved in order that i)-ii) take place always has a non-trivial solution. Notice that when $w_{\bf n} \equiv 1$ we recover the definition given above for classical type I Hermite-Pad\'e approximation.

An analogous definition can be given for type II multi-point Hermite-Pad\'e approximants but we will not dwell into this. Algebraic and analytic properties regarding uniqueness, integral representations, asymptotic behavior, and orthogonality conditions satisfied by  type I and type II Hermite-Pad\'e approximants have been studied, for example,  in \cite{BL}, \cite{DrSt0}, \cite{DrSt1}, \cite{DrSt2}, \cite{FL1}, \cite{FL2}, \cite{FL3}, \cite{FL4}, \cite{GRS}, \cite{Nik0}, and \cite[Chapter 4]{NS}, which include the case of multi-point approximation.

Let $\stackrel{\circ}{\Delta}$ denote the interior of $\Delta$ with the Euclidean topology of the real line. We have

\begin{theorem} \label{unicidad} Let $(s_{1,1},\ldots,s_{1,m}) = \mathcal{N}(\sigma_1,\ldots,\sigma_m), {\bf n} = (n_1,\ldots,n_m) \in \mathbb{Z}_+^{m} \setminus \{{\bf 0}\},$ and $w_{\bf n},$ $ \deg w_{\bf n} \leq |{\bf n}| + \max(n_j)-2,$ a polynomial with real coefficients whose zeros lie in $\mathbb{C} \setminus\Delta_1$, be given. The type I multi-point Hermite-Pad\'e approximation $\left(a_{{\bf n},0}, \ldots, a_{{\bf n},m}\right)$ of $(\widehat{s}_{1,1},\ldots,\widehat{s}_{1,m})$ with respect to $w_{\bf n}$ is uniquely determined except for a constant factor, and $\deg a_{{\bf n},j} = n_j -1, j=0,\ldots,m$. Moreover
\begin{equation} \label{ortog}
\int x^{\nu} \mathcal{A}_{{\bf n},1}(x) \frac{d\sigma_1(x)}{w_{\bf n}(x)} = 0, \qquad \nu = 0,\ldots,|{\bf n}| -2,
\end{equation}
which implies that $\mathcal{A}_{{\bf n},1}$ has exactly $|{\bf n}| -1$ simple zeros in  $\stackrel{\circ}{\Delta}_1$ and no other zeros in $\mathbb{C} \setminus \Delta_2$. Additionally,
\begin{equation} \label{resto}
\frac{\mathcal{A}_{{\bf n},0}(z)}{w_{\bf n}(z)} = \int \frac{ \mathcal{A}_{{\bf n},1}(x) d\sigma_1(x)}{w_{\bf n}(x)(z-x)}
\end{equation}
and
\begin{equation} \label{a0} a_{{\bf n},0}(z) = - \int \frac{\sum_{j=1}^{m} (w_{\bf n}(x) a_{{\bf n},j}(z) - w_{\bf n}(z) a_{{\bf n},j}(x)) d{s}_{1,j}(x) }{(z-x) w_{\bf n}(x)}.
\end{equation}
\end{theorem}

Notice that nothing has been said about the location of the zeros of the polynomials $a_{{\bf n},j}$. For special sequences of multi-indices this information can be deduced from the convergence of type I Hermite-Pad\'e approximants. We have the following result (see also Lemma \ref{haus}).

\begin{theorem} \label{converge}  Let $S= (s_{1,1},\ldots,s_{1,m}) = \mathcal{N}(\sigma_1,\ldots,\sigma_m)$, $\Lambda \subset \mathbb{Z}_+^{m}$ an infinite sequence of distinct muti-indices, and  $(w_{\bf n})_{{\bf n}\in \Lambda}, \deg w_{\bf n} \leq |{\bf n}| + \max(n_j)-2,$ a sequence of  polynomials with real coefficients whose zeros lie in $\mathbb{C} \setminus\Delta_1$, be given. Consider the corresponding sequence $\left(a_{{\bf n},0},\ldots, a_{{\bf n},m}\right), {\bf n} \in \Lambda,$ of  type I multi-point Hermite-Pad\'e approximants of $S$ with respect to $(w_{\bf n})_{{\bf n} \in \Lambda}$. Assume that
\begin{equation}\label{cond1} \sup_{{\bf n}\in \Lambda}\left(\max_{j=1,\ldots,m}(n_j) - \min_{k=1,\ldots,m}(n_k) \right)\leq C < \infty, \end{equation}
and that either $\Delta_{m-1}$ is bounded away from $\Delta_m$  or $\sigma_m$ satisfies Carleman's condition.
Then,
\begin{equation} \label{fund1} \lim_{{\bf n} \in \Lambda}  \frac{a_{{\bf n}, j}}{a_{{\bf n},m}} =  (-1)^{m-j}\widehat{s}_{m,j+1},\qquad j=0,\ldots,m-1,
\end{equation}
uniformly on each compact subset $\mathcal{K} \subset \mathbb{C} \setminus \Delta_m$. The accumulation points of   sequences of zeros of the polynomials $a_{{\bf n},j}, j=0,\ldots,m, {\bf n} \in \Lambda$ are contained in $\Delta_m \cup \{\infty\}$. Additionally,
\begin{equation} \label{fund2} \lim_{{\bf n} \in \Lambda}  \frac{\mathcal{A}_{{\bf n},j}}{a_{{\bf n},m}} = 0, \qquad j=0,\ldots,m-1,
\end{equation}
uniformly on each compact subset $\mathcal{K} \subset \mathbb{C} \setminus (\Delta_{j+1} \cup \Delta_m)$.
\end{theorem}

We wish to underline that  Theorem \ref{converge} requires no special analytic property  from the generating measures of the Nikishin system except for Carleman's condition on  $\sigma_m$.

Notice that the sequences of rational functions $\left( {a_{{\bf n}, j}}/{a_{{\bf n},m}}\right), {\bf n} \in \Lambda, j=0,\ldots,m-1,$ allow to recover the Cauchy transforms of the measures in the Nikishin system $\mathcal{N}(\sigma_m,\ldots,\sigma_1)$ in contrast with the sequences $\left( {P_{{\bf n}, j}}/{Q_{{\bf n}}}\right), {\bf n} \in \Lambda, j=1,\ldots,m,$ of type II multi-point Hermite-Pad\'e approximants which recover the Cauchy transforms of the measures in  $\mathcal{N}(\sigma_1,\ldots,\sigma_m)$.

In the process of writing this paper, S.P. Suetin sent us \cite{RS1} and \cite{RS2}. The first one of these papers announces the results contained in the second one. Those papers deal with the study of type I Hermite-Pad\'e approximants for an interesting class of systems  of two functions $(m=2)$ which form a generalized Nikishin system  in the sense that the second generating measure lives on a symmetric (with respect to the real line) compact set which does not separate the complex plane and is made up of finitely many analytic arcs. The authors obtain the logarithmic asymptotic of the sequences of Hermite-Pad\'e polynomials   $a_{{\bf n},j}, j=1,2,$ and an analogue of \eqref{fund1} for $j=1$. Convergence is proved in capacity (see \cite[Theorem 1]{RS1} and \cite[Corollary 1]{RS2}.

For the proof of Theorem \ref{converge} we need a convenient representation of the reciprocal  of the Cauchy transform of a measure. It is  known that for each $\sigma
\in {\mathcal{M}}(\Delta),$ where $\Delta$ is contained in a half line, there exists a measure $\tau \in
{\mathcal{M}}(\Delta)$ and ${\ell}(z)=a z+b, a = 1/|\sigma|, b \in {\mathbb{R}},$ such that
\begin{equation} \label{s22}
{1}/{\widehat{\sigma}(z)}={\ell}(z)+ \widehat{\tau}(z),
\end{equation}
where $|\sigma|$ is the total variation of $\sigma.$  See  \cite[Appendix]{KN} and \cite[Theorem 6.3.5]{stto} for measures with compact support, and \cite[Lemma 2.3]{FL4} when the support is contained in a half line.

We call $\tau$  the inverse measure of $\sigma.$ Such measures appear frequently in our reasonings, so we will fix a
notation to distinguish them. In relation with measures  denoted with $s$ they will carry over to them the
corresponding sub-indices. The same goes for the  polynomials
$\ell$. For example,
\[
{1}/{\widehat{s}_{j,k}(z)}  ={\ell}_{j,k}(z)+
\widehat{\tau}_{j,k}(z).
\]
We also write
\[
{1}/{\widehat{\sigma}_{\alpha}(z)} ={\ell}_{\alpha }(z)+
\widehat{\tau}_{\alpha }(z).
\]

The following result has independent interest and will be used in combination with Lemma \ref{BusLop} below in the  proof of Theorem \ref{converge}.

\begin{theorem} \label{momentos} Let $(s_{1,1},s_{1,2}) = \mathcal{N}(\sigma_1,\sigma_2)$.   If $\sigma_1$ satisfies Carleman's condition so do $s_{1,2}$ and $\tau_1$.
\end{theorem}

This paper is organized as follows. In Section \ref{aux} we prove Theorems \ref{unicidad} and \ref{momentos}. We also present some notions and results necessary for the proof of Theorem \ref{converge}. Section \ref{proofmain} contains the proof of Theorem \ref{converge} and  some extensions of the main result to sequences of multi-indices satisfying conditions weaker than \eqref{cond1}, estimates of the rate of convergence in \eqref{fund1}-\eqref{fund2} for the case when $\Delta_m$ or $\Delta_{m-1}$ is bounded and $\Delta_m \cap \Delta_{m-1} = \emptyset$, and  applications  to other simultaneous approximation schemes.

\section{Proof of Theorem \ref{ortog} and auxiliary results} \label{aux}

We begin with a lemma which allows to give an integral representation of the remainder of type I multi-point  Hermite-Pad\'e approximants.

\begin{lemma} \label{reduc} Let $(s_{1,1},\ldots,s_{1,m}) = \mathcal{N}(\sigma_1,\ldots,\sigma_m)$ be given. Assume that there exist polynomials with real coefficients $a_0,\ldots,a_m$ and a polynomial $w$ with real coefficients whose zeros lie in $\mathbb{C} \setminus \Delta_1$  such that
\[\frac{\mathcal{A}(z)}{w(z)} \in \mathcal{H}(\mathbb{C} \setminus \Delta_1)\qquad \mbox{and} \qquad \frac{\mathcal{A}(z)}{w(z)} = \mathcal{O}\left(\frac{1}{z^N}\right), \quad z \to \infty,
\]
where $\mathcal{A}  := a_0 + \sum_{k=1}^m a_k  \widehat{s}_{1,k} $ and $N \geq 1$. Let $\mathcal{A}_1  := a_1 + \sum_{k=2}^m a_k  \widehat{s}_{2,k} $. Then
\begin{equation} \label{eq:3}
\frac{\mathcal{A}(z)}{w(z)} = \int \frac{\mathcal{A}_1(x)}{(z-x)} \frac{d\sigma_1(x)}{w(x)}.
\end{equation}
If $N \geq 2$, we also have
\begin{equation} \label{eq:4}
\int x^{\nu}  \mathcal{A}_1(x)  \frac{d\sigma_1(x)}{w(x)}, \qquad \nu = 0,\ldots, N -2.
\end{equation}
In particular, $\mathcal{A}_1$ has at least $N -1$ sign changes in  $\stackrel{\circ}{\Delta}_1 $.
\end{lemma}

\begin{proof} We have
\[ \mathcal{A}(z) = a_0(z) + \sum_{k=1}^m a_{k}(z)\widehat{s}_{1,k}(z) \mp w(z) \int \frac{\mathcal{A}_1(x)}{(z-x)} \frac{d\sigma_1(x)}{w(x)} =
\]
\[ a_0(z) + \int \frac{\sum_{k=1}^{m} (w (x) a_{k}(z) - w (z) a_{k} (x)) d{s}_{1,k}(x) }{(z-x) w (x)} +
 w (z) \int \frac{\mathcal{A}_1(x)}{(z-x)} \frac{d\sigma_1(x)}{w(x)} .
\]
For each $k=1,\ldots,m$
\[\left( { w (x) a_{ k}(z) - w (z) a_{ k}(x) }\right)/{(z-x)}
\]
is a polynomial in $z$. Therefore,
\[ P (z)  := a_0(z) + \int \frac{\sum_{k=1}^{m} (w (x) a_{k}(z) - w (z) a_{k} (x)) d{s}_{1,k}(x) }{(z-x) w (x)}
\]
represents a polynomial. Consequently
\[ \mathcal{A}(z) = P (z) + w(z) \int \frac{\mathcal{A}_1(x) d\sigma_1(x)}{(z-x)w (x)} =
w(z) \mathcal{O}(1/z^{N}), \quad z\to \infty.
\]
These equalities imply that
\[  P (z) = w(z)\mathcal{O}(1/z), \qquad z \to \infty,
\]
Therefore, $\deg P < \deg w$ and is equal to zero at all the zeros of $w$. Hence $P \equiv 0$. (Should $w$ be a constant polynomial likewise we get that $P \equiv 0$.) Thus, we have proved \eqref{eq:3}.

From our assumptions and \eqref{eq:3}, it follows that
\[ \frac{\mathcal{A}(z)}{w(z)} = \int \frac{\mathcal{A}_1(x)}{(z-x)} \frac{d\sigma_1(x)}{w(x)} = \mathcal{O}(1/z^{N}), \qquad z \to \infty.
\]
Suppose that $N \geq 2$.  We have the asymptotic expansion
\[ \int \frac{\mathcal{A}_1(x)}{(z-x)} \frac{d\sigma_1(x)}{w(x)} =
\]
\[  \sum_{\nu=0}^{N -2} \frac{d_{ \nu}}{z^{\nu +1}} + \int \frac{x^{N-1}\mathcal{A}_1(x)}{z^{N -1}(z-x) } \frac{d \sigma_1(x)}{w (x)}= \sum_{\nu=0}^{N-2} \frac{d_{ \nu}}{z^{N +1}} + \mathcal{O}( {1}/{z^{N}}), \quad z \to \infty,
\]
where
\[ d_{ \nu} = \int x^{\nu}  \mathcal{A}_1(x)   \frac{d\sigma_1(x)}{ w (x)} ,\qquad \nu =0,\ldots,N -2.
\]
Therefore,
\[ d_{ \nu} = 0, \qquad \nu=0,\ldots,N -2,
\]
which is \eqref{eq:4}.

Suppose that $\mathcal{A}_1$ has at most
$\widetilde{N} < N -1$ sign changes in  $\stackrel{\circ}{\Delta}_1 $ at the points $x_1,\ldots, x_N$. Take $q(x) = \prod_{k=1}^{\widetilde{N}} (x-x_k)$. According to $\eqref{eq:4}$
\[ \int q(x) \mathcal{A}_1(x)  \frac{d\sigma_1(x)}{ w (x)} = 0
\]
which is absurd because $q (a_1  + \sum_{k=2}^m a_k  \widehat{s}_{2,k}) /w  $ has constant sign in $\stackrel{\circ}{\Delta}_1 $ and $\sigma_1$ is a measure with constant sign in $\stackrel{\circ}{\Delta}_1 $ whose support contains infinitely many points. Thus, the number of sign changes must be greater or equal to $N -1$ as claimed.
\end{proof}

In \cite[Lemma 2.10]{FL4}, several formulas involving ratios of Cauchy transforms were proved. The most useful ones  in this paper establish that
\begin{equation} \label{4.4}
\frac{\widehat{s}_{1,k}}{\widehat{s}_{1,1}} =
\frac{|s_{1,k}|}{|s_{1,1}|} - \langle \tau_{1,1},\langle s_{2,k},\sigma_1
\rangle \widehat{\rangle}  , \qquad  1=j < k \leq m.
\end{equation}

We are ready for the

\begin{proof}[Proof of Theorem \ref{unicidad}] Let $(a_{{\bf n},0},\ldots,a_{{\bf n},m})$ be a  type I multi-point Hermite-Pad\'e approximation of $(\widehat{s}_{1,1},\ldots,\widehat{s}_{1,m})$ with respect to $w_{\bf n}$. From Definition \ref{type1}, formulas \eqref{ortog} and \eqref{resto} follow directly from \eqref{eq:4} and \eqref{eq:3}, respectively. Relation \eqref{a0} is obtained from \eqref{resto} solving for $a_{{\bf n},0}$.

In the proof of Lemma  \ref{reduc} we saw that \eqref{ortog} implies that  $\mathcal{A}_{{\bf n},1}$ has at least $|{\bf n}|-1$ sign changes in  $\stackrel{\circ}{\Delta}_1 $.
We have that $(s_{2,2},\ldots,s_{2,m}) = \mathcal{N}(\sigma_2,\ldots,\sigma_m)$ forms a Nikishin system. According to \cite[Theorem 1.1]{FL4},  $\mathcal{A}_{{\bf n},1}$ can have at most $|{\bf n}| -1$ zeros in $\mathbb{C} \setminus \Delta_2$. Taking account of what we proved  previously, it follows that $\mathcal{A}_{{\bf n},1}$ has exactly $|{\bf n}|-1$ simple zeros in  $\stackrel{\circ}{\Delta}_1$  and it has no other zero in $\mathbb{C} \setminus \Delta_2$. This is true for any ${\bf n} \in \mathbb{Z}_+^m  \setminus \{\bf 0\}$.

Suppose that for some ${\bf n} \in \mathbb{Z}_+^m  \setminus \{\bf 0\}$  and some $j \in \{1,\ldots,m\}$, we have that $\deg a_{{\bf n},j} = \widetilde{n}_j -1 < n_j -1$. Then, according to \cite[Theorem 1.1]{FL4}  $\mathcal{A}_{{\bf n},1}$ could have at most $|{\bf n}| - n_j + \widetilde{n}_j -1 \leq |{\bf n}| -2$ zeros in $\mathbb{C} \setminus \Delta_2$. This is absurd because we have proved that it   has $|{\bf n}| -1$ zeros in  $\stackrel{\circ}{\Delta}_1$.

Now, suppose that for some ${\bf n} \in \mathbb{Z}_+^m \setminus \{\bf 0\}$, there exist two non collinear type I multi-point Pad\'e approximants  $\left(a_{{\bf n},0}, \ldots, a_{{\bf n},m}\right)$ and $\left(\widetilde{a}_{{\bf n},0}, \ldots, \widetilde{a}_{{\bf n},m}\right)$  of $(\widehat{s}_{1,1},\ldots,\widehat{s}_{1,m})$ with respect to $w_{\bf n}$. From \eqref{a0} it follows that $\left(a_{{\bf n},1}, \ldots, a_{{\bf n},m}\right)$ and $\left(\widetilde{a}_{{\bf n},1}, \ldots, \widetilde{a}_{{\bf n},m}\right)$ are not collinear. We know that $\deg a_{{\bf n},j} = \deg \widetilde{a}_{{\bf n},j} = n_j -1, j=1,\ldots,m$. Consequently, there exists some constant $C$ such that $\left({a}_{{\bf n},1}-C\widetilde{a}_{{\bf n},1}, \ldots, {a}_{{\bf n},m}-C \widetilde{a}_{{\bf n},m}\right) \neq {\bf 0}$ and  $\deg(a_{{\bf n},{j}} - C\widetilde{a}_{{\bf n}, {j}}) < n_{ {j}} -1$ for some $j \in \{1,\ldots,m\}$. By linearity, $\left({a}_{{\bf n},0}-C\widetilde{a}_{{\bf n},0}, \ldots, {a}_{{\bf n},m}-C \widetilde{a}_{{\bf n},m}\right)$ is a multi-point type I Hermite-Pad\'e approximant of $(\widehat{s}_{1,1},\ldots,\widehat{s}_{1,m})$ with respect to $w_{\bf n}$. This is not possible because $\deg(a_{{\bf n},{j}} - C\widetilde{a}_{{\bf n}, {j}}) < n_{ {j}} -1$. Therefore, non-collinear solutions cannot exist.

We still need to show that $\deg a_{{\bf n}_0} = n_0-1$. To this end we need to transform $\mathcal{A}_{{\bf n},0}$. Let $j$ be the first component of $\bf n$ such that $n_j = \max_{k=1,\ldots,m} n_k $. Since $n_0 = n_j -1$, we have that either $j=1$ or $n_0 \geq n_k, k=1,\ldots,j-1$. If $j=1$, using \eqref{s22} and \eqref{4.4} it follows that
\[ \mathcal{B}_{{\bf n},0} := \frac{\mathcal{A}_{{\bf n},0}}{\widehat{s}_{1,1}} = \ell_{1,1} a_{{\bf n},0} + \sum_{k=1}^m \frac{|s_{1,k}|}{|s_{1,1}|} a_{{\bf n},k} + a_{{\bf n},0} \widehat{\tau}_{1,1} - \sum_{k=2}^m a_{{\bf n},k} \langle \tau_{1,1}, \langle s_{2,k},\sigma_1 \rangle \widehat{\rangle},
\]
where
\[  \mathcal{B}_{{\bf n},0}/w_{\bf n} \in \mathcal{H}(\mathbb{C} \setminus \Delta_1), \qquad  \mathcal{B}_{{\bf n},0}(z)/w_{\bf n}(z)= \mathcal{O}(1/z^{|{\bf n}|-1}), \quad z\to \infty.
\]
Using Lemma \ref{reduc} it follows that
\[ \int x^{\nu}  \mathcal{B}_{{\bf n},1}(x)  \frac{d\tau_{1,1}(x)}{w_{\bf n}(x)}, \qquad \nu = 0,\ldots, |{\bf n}|-3.
\]
where $\mathcal{B}_{{\bf n},1} = a_{{\bf n},0} - \sum_{k=2}^m a_{{\bf n},k} \langle   \langle\sigma_2,\sigma_1 \rangle,\sigma_3,\ldots,\sigma_k \widehat{\rangle}$. Hence $\mathcal{B}_{{\bf n},1}$ has at least $|{\bf n}|-2$ sign changes in $\stackrel{\circ}{\Delta}_1$. According to \cite[Theorem 1.1]{FL4} the linear form $\mathcal{B}_{{\bf n},1}$ has at most $\deg a_{{\bf n},0} + n_2+\cdots+ n_m$ zeros in all of $\mathbb{C} \setminus \Delta_2$. Should  $\deg a_{{\bf n},0} \leq n_0-2$,   we would have that $\deg a_{{\bf n},0} + n_2+\cdots+ n_m \leq |{\bf n}|-3$ which contradicts that $\mathcal{B}_{{\bf n},1}$  has at least $|{\bf n}|-2$ zeros in $\stackrel{\circ}{\Delta}_1$. Thus, when $j=1$ it is true that $\deg a_{{\bf n},0} = n_0 -1$. In general, the proof is similar as we will see.

Suppose that $j$, as defined in the previous paragraph, is $\geq 2$. Then, either $n_0 = n_k, k=1,\ldots,{j-1}$ or there exists $\overline{\jmath} < j$ for which $n_0 = n_k, k=1,\ldots,{\overline{\jmath}-1}$ and $n_0 > n_{\overline{\jmath}}$. In the first case, applying \cite[Lemma 2.12]{FL4},  we obtain that there exists a Nikishin system $(s_{1,1}^*,\ldots,s_{1,m}^*) =
{\mathcal{N}}(\sigma_1^*,\ldots,\sigma_m^*)$, a multi-index $\,{\bf n}^* = (n_0^*,\ldots,n_m^*) \in {\mathbb{Z}}_+^{m+1}$
which is a permutation of ${\bf n}$ with $n_0^* = n_j$, and polynomials with real coefficients $a_{{\bf n},k}^*, \deg a_{{\bf n},k}^* \leq n_k^* -1, k=0,\ldots,m$, such that
\[ \frac{\mathcal{A}_{{\bf n},0}}{\widehat{s}_{1,j}} =  a_{{\bf n},0}^* + \sum_{k=1}^m a_{{\bf n},k}^*. \widehat{s}_{1,k}^*
\]
Due to the structure of the values of the components of the multi-index $a_{{\bf n},j}^* = (-1)^ja_{{\bf n},0}$ and $ n_j^* = n_0$ (see formula (31) in \cite{FL3}). We can proceed as before and find that $\deg a_{{\bf n},j}^* = n_j^* -1, j=1,\ldots,m$. In particular, $\deg a_{{\bf n},j} = n_0 -1$. In the other case, \cite[Lemma 2.12]{FL4} gives that
\[ \frac{\mathcal{A}_{{\bf n},0}}{\widehat{s}_{1,j}} =  a_{{\bf n},0}^* + \sum_{k=1}^m a_{{\bf n},k}^*, \widehat{s}_{1,k}^*
\]
where $a_{{\bf n}\overline{\jmath}}^* = \pm a_{{\bf n},0} + C a_{{\bf n},\overline{\jmath}},$  $C\neq 0$ is some constant, and $n_{\overline{\jmath}}^* = n_0$ (see formula (31) in \cite{FL3}). Repeating the arguments employed
above, we obtain that $\deg a_{{\bf n},j}^* = n_j^* -1, j=1,\ldots,m$. In particular, $\deg a_{{\bf n},0} = n_0 -1$. because we already know that $\deg a_{{\bf n},\overline{\jmath}} = n_{\overline{\jmath}} -1< n_0 -1$.
\end{proof}

\begin{remark} We wish to point out that in the statement of \cite[Theorem 1.1]{FL4} there is a missprint on the last line where $ {\mathbb{C}}$ should replace $\overline{\mathbb{C}}$.  That is, it should refer to zeros at finite points. This can be checked looking at the statements of \cite[Lemmas 2.1, 2.2]{FL4} and the proof of \cite[Theorem 1.1]{FL4} itself.
\end{remark}

The notion of convergence in Hausdorff content plays a central  role in the proof of Theorem \ref{converge}. Let $B$ be a subset of the complex plane $\mathbb{C}$. By
$\mathcal{U}(B)$ we denote the class of all coverings of $B$ by at
most a numerable set of disks. Set
$$
h(B)=\inf\left\{\sum_{i=1}^\infty
|U_i|\,:\,\{U_i\}\in\mathcal{U}(B)\right\},
$$
where $|U_i|$ stands for the radius of the disk $U_i$. The quantity
$h(B)$ is called the $1$-dimensional Hausdorff content of the
set $B$.

Let $(\varphi_n)_{n\in\mathbb{N}}$ be a sequence of complex functions
defined on a domain $D\subset\mathbb{C}$ and $\varphi$ another
function defined on $D$ (the value $\infty$ is permitted). We say that
$(\varphi_n)_{n\in\mathbb{N}}$ converges in Hausdorff content to
the function $\varphi$ inside $D$ if for each compact
subset $\mathcal{K}$ of $D$ and for each $\varepsilon
>0$, we have
$$
\lim_{n\to\infty} h\{z\in K :
|\varphi_n(z)-\varphi(z)|>\varepsilon\}=0
$$
(by convention $\infty \pm \infty = \infty$). We denote this writing $h$-$\lim_{n\to \infty} \varphi_n =
\varphi$ inside $D$.

To obtain Theorem \ref{converge} we first prove \eqref{fund1} with convergence in Hausdorff content in place of uniform convergence (see Lemma \ref{haus} below). We need the following notion.

Let $s \in \mathcal{M}(\Delta)$ where $\Delta$ is contained in a half line of the real axis. Fix an arbitrary $\kappa \geq -1$. Consider a sequence of polynomials $(w_n)_{n \in \Lambda},  \Lambda \subset \mathbb{Z}_+,$ such that $\deg w_n = \kappa_n \leq 2n + \kappa +1$, whose zeros lie in $\mathbb{R} \setminus \Delta$. Let $(R_n)_{n \in \Lambda}$ be a sequence of rational functions $R_n = p_n/q_n$ with real coefficients satisfying the following conditions for each $n \in \Lambda$:
\begin{itemize}
\item[a)] $\deg p_n \leq n + \kappa,\quad  \deg q_n \leq n, \quad q_n \not\equiv 0,$
\item[b)] $ {(q_n \widehat{s} - p_n)(z)}/{w_n} = \mathcal{O}\left( {1}/{z^{n+1 - \ell}}\right) \in \mathcal{H}(\mathbb{C}\setminus \Delta), z \to \infty,$ where $\ell \in \mathbb{Z}_+$ is fixed.
\end{itemize}
We say that $(R_n)_{n \in \Lambda}$ is a sequence of incomplete diagonal multi-point Pad\'e approximants of $\widehat{s}$.

Notice that in this construction for each $n \in \Lambda$ the number of free parameters equals $2n + \kappa +2$ whereas the number of homogeneous linear equations to be solved in order to find $q_n$ and $p_n$ is equal to $2n + \kappa - \ell + 1$. When $\ell =0$ there is only one more parameter than equations and  $R_n$ is defined uniquely coinciding with a (near) diagonal multi-point Pad\'e approximation. When $\ell \geq 1$ uniqueness is not guaranteed, thus the term incomplete.

For sequences of incomplete diagonal multi-point Pad\'e approximants, the following Stieltjes type theorem was proved in \cite[Lemma 2]{BL} in terms of convergence in Hausdorff content.

\begin{lemma} \label{BusLop}
Let $s \in \mathcal{M}(\Delta)$ be given where $\Delta$ is contained in a half line. Assume that $(R_n)_{n \in \Lambda}$ satisfies {\rm a)-b)} and either the number of zeros of $w_n$ lying on a closed bounded segment of $\mathbb{R} \setminus \Delta$ tends to infinity as $n\to\infty, n \in \Lambda$, or $s$ satisfies Carleman's condition.
Then
\[ h-\lim_{n \in \Lambda} R_n = \widehat{s},\qquad \mbox{inside}\qquad  \mathbb{C} \setminus \Delta.
\]
\end{lemma}

We will need to use Lemma \ref{BusLop} for different measures and Theorem \ref{momentos} comes in our aid.

\begin{proof}[Proof of Theorem \ref{momentos}] Without loss of generality, we can assume that $\Delta \subset \mathbb{R}_+$ and that $\sigma_1$ is positive.
Let $(c_n)_{n \in \mathbb{Z}_+}$ and $(\widetilde{c}_n)_{n \in \mathbb{Z}_+}$ denote the sequences of moments of $\sigma_1$ and $s_{1,2}$, respectively. Since $\widehat{\sigma}_2$ has constant sign on $\mathbb{R}_+$, we have that
\[ |\widetilde{c}_n| = \int x^{n} |\widehat{\sigma}_2(x)| d\sigma_1(x) \leq \int_0^1 x^{n} |\widehat{\sigma}_2(x)| d\sigma_1(x) + \int_1^\infty x^{n} |\widehat{\sigma}_2(x)| d\sigma_1(x)  \leq
 |s_{1,2}| + C c_n,
\]
where $C = \max \{|\widehat{\sigma}_2(x)|: x \in [1,+\infty)\} < \infty$ because $\lim_{x\to \infty}\widehat{\sigma}_2(x) =0$. Consequently,
\[ \sum_{n \geq 1} |\widetilde{c}_n|^{-1/2n} \geq \sum_{n \geq 1} (|s_{1,2}| + C c_n )^{-1/2n} \geq
\sum_{\{n:C c_n < |s_{1,2}|\}}(2|s_{1,2}|)^{-1/2n} + \sum_{\{n:C c_n \geq |s_{1,2}|\}}(2C c_n  )^{-1/2n}.
\]
If the first sum after the last inequality contains infinitely many terms then that sum is already divergent. If it has finitely many terms then Carleman's condition for $\sigma_1$ guarantees that the second sum is divergent. Thus, $s_{1,2}$ satisfies Carleman's condition.

To prove the second part we need to express the moments $(d_n)_{n\in \mathbb{Z}_+} $ of $\tau_1$ in terms of the moments of $\sigma_1$. In the proof of \cite[Lemma 2.3]{FL4} we showed that the moments $(d_n)_{n \in \mathbb{Z}_+}$ are finite (since all the moments of $\sigma_1$ are finite) and  
can be obtained solving  the system of equations
\[
\begin{array}{ccl}
1 & = & d_{-2} c_0 \\
0 & = & d_{-2} c_1 + d_{-1} c_0\\
0 & = & d_{-2} c_2 + d_{-1} c_1 + d_0 c_0\\
\vdots & = & \vdots  \\
0 & = & d_{-2} c_{n+2} + d_{-1} c_{n+1} + \cdots + d_n c_0\,\,.
\end{array}
\]
(The values of $d_{-2}$ and $d_{-1}$ turn out to be the coefficients $a$ and $b$, respectively, of the polynomial $\ell_1$ in the decomposition \eqref{s22} of $1/\widehat{\sigma}_1$.) Read the paragraph after formula (9) in \cite{FL4}.

To find $d_n$ we apply Cramer's rule and we get
\begin{equation} \label{dn} d_n = (-1)^{n}\Omega_n/c_0^{n+3}
\end{equation}
where $c_0^{n+3}$ gives the value of the determinant of the system and
\[ \Omega_n = \left|
\begin{array}{cccc}
c_1 & c_0 & 0 & \cdots \\
c_2 & c_1 & \ddots & \ddots \\
\vdots & \ddots & \ddots & \ddots \\
c_{n+2}& c_{n+1} & \cdots & c_1
\end{array}
\right|
\]
is the determinant of a  lower Hessenberg matrix of dimension $n+2$ with constant diagonal terms. The expansion of the determinant $\Omega_n$ has several characteristics:
\begin{itemize}
\item It has exactly $2^{n+1}$ non zero terms.
\item For each $n \geq 0$, the sum of the subindexes of each non zero term equals $n+2$ (if a factor is repeated its subindex is counted as many times as it is repeated).
\item The number of factors in each term is equal to $n+2$.
\end{itemize}

The last assertion is trivial.
To calculate the number of non zero terms notice that from the first row we can only choose 2 non zeros entries. Once this is done, from the second row we can only choose 2 non zero entries, and so forth, until we get to the last row where we only have left one non zero entry to choose.

Regarding the second assertion we use induction. When $n=0$ it is obvious. Assume that each non zero term in the expansion of $\Omega_n$ has the property that the sum of its subindexes equals $n+2$  and let us show that each non zero term in the expansion of $\Omega_{n+1}$ has the property that the sum of its subindexes equals $n+3$. Expanding $\Omega_{n+1}$ by its first row we have
\[ \Omega_{n+1} = c_1 \Omega_n - c_0 \Omega_{n}^*,
\]
where $\Omega_n^*$ is obtained substituting the first column of $\Omega_n$ by the column vector $(c_2,\ldots,c_{n+3})^t$ (the superscript $t$ means taking transpose). Using the induction hypothesis it   easily follows that in each term arising from $c_1 \Omega_n$ and $c_0 \Omega_{n}^*$ the sum of its subindexes must equal $n+3$.

Using the properties proved above we obtain that the general expression of $\Omega_n$ is
\[ \Omega_n = \sum_{j=1}^{n+2 }\sum_{\alpha_1 + \cdots + \alpha_{j} = n+2} \varepsilon_{\alpha} c_0^{n+2-j}c_{\alpha_1}\cdots c_{\alpha_{j}},
\]
where $\alpha = (\alpha_1, \ldots, \alpha_{j}), 1 \leq \alpha_k \leq n+2, 1\leq k \leq j$ and $\varepsilon_{\alpha} = \pm1$. Thus
\begin{equation} \label{omegan} |\Omega_n| \leq \sum_{j=1}^{n+2 }\sum_{\alpha_1 + \cdots + \alpha_{j} = n+2}  c_0^{n+2-j}c_{\alpha_1}\cdots c_{\alpha_{j}}.
\end{equation}
In this sum, there is  there is only one term which contains the factor $c_{n+2}$ and that is when $j=1$. That term is $c_0^{n+1}c_{n+2}$. In the rest of the terms $1 \leq {\alpha_k} \leq n+1$. Let us prove that
\begin{equation} \label{cota} c_0^{n+2-j}c_{\alpha_1}\cdots c_{\alpha_{j}} \leq c_0^{n+1}c_{n+2}\qquad \mbox{for all} \qquad \alpha.
\end{equation}
In fact, using the Holder inequality on each factor except the first, it follows that
\[ c_0^{n+2-j}c_{\alpha_1}\cdots c_{\alpha_{j}} \leq c_0^{n+2-j} \left(\int x^{n+2} d\sigma_1(x)\right)^{\sum_{k=1}^j \alpha_k/(n+2)}  \left(\int  d\sigma_1(x)\right)^{j - (\sum_{k=1}^j \alpha_k)/(n+2)}.
\]
It remains to employ that $\sum_{k=1}^j \alpha_k=n+2$ to complete the proof of \eqref{cota}.

From \eqref{dn}, \eqref{omegan}, and \eqref{cota},  we have that
\[ d_n \leq 2^{n+1}c_{n+2}/c_0^2
\]
and the Carleman condition for $\tau_1$ readily follows. \end{proof}

An immediate consequence of Theorem \ref{momentos} is the following
\begin{corollary}  \label{momentosNik} Let $(s_{1,1},\ldots,s_{1,m}) = \mathcal{N}(\sigma_1,\ldots,\sigma_m)$ be such that $\Delta_1$ is contained in a half line and $\sigma_1$ satisfies Carleman's condition. Then, for all $j=1,\ldots,m$ we have that $s_{1,j}$ and $\tau_{1,j}$ satisfies Carleman's condition.
\end{corollary}

\begin{proof} For $s_{1,1}$ the assertion is the hypothesis. Let $j\in \{2,\ldots,m\}$. Notice that $s_{1,j} = \langle \sigma_1, s_{2,j} \rangle$ and $(s_{1,1},s_{1,j}) = \mathcal{N}(\sigma_1,s_{2,j})$ so $s_{1,j}, j=2,\ldots,n$ satisfies Carleman's condition due to   Theorem \ref{momentos}. Since $s_{1,j},j=1,\ldots,m$ satisfies Carleman's condition then Theorem \ref{momentos} also gives that $\tau_{1,j},  j=1,\ldots,m$ satisfy Carleman's condition.
\end{proof}

Actually we will use this result for $(s_{m,m},\ldots,s_{m,1}) = \mathcal{N}(\sigma_m,\ldots,\sigma_1)$.

\section{Proof of Theorem \ref{converge}} \label{proofmain}

The first step consists in proving a weaker version of \eqref{fund1}.

\begin{lemma} \label{haus} Assume that the conditions of Theorem \ref{converge} are fulfilled. Then,
for each fixed $j=0,\ldots,m-1$
\begin{equation} \label{convHaus}
h-\lim_{n\in \Lambda}\frac{a_{{\bf n}, j}}{a_{{\bf n},m}} = (-1)^{m-j}\widehat{s}_{m,j+1}, \quad h-\lim_{n\in \Lambda}\frac{a_{{\bf n}, m}}{a_{{\bf n},j}} = (-1)^{m-j}\widehat{s}_{m,j+1}^{-1}\quad \mbox{inside} \quad \mathbb{C} \setminus \Delta_m.
\end{equation}
There exists a constant $C_1$, independent of $\Lambda$, such that for each $j=0,\ldots,m$ and ${\bf n} \in \Lambda,$ the polynomials $a_{{\bf n},j}$  have at least $(|{\bf n}|/m) - C_1$ zeros in $\stackrel{\circ}{\Delta}_j $.
\end{lemma}

\begin{proof} If $m=1$ the statement reduces directly to Lemma \ref{BusLop}, so without loss of generality we can assume that $m \geq 2$. Fix ${\bf n} \in \Lambda$.

In Theorem \ref{unicidad} we proved that $\mathcal{A}_{{\bf n},1}$ has exactly $|{\bf n}| -1$ simple zeros in $\mathbb{C} \setminus \Delta_2$ and they all lie in  $\stackrel{\circ}{\Delta}_1 $. Therefore, there exists a polynomial
$w_{{\bf n},1}, \deg w_{{\bf n},1} = |{\bf n}| -1,$ whose zeros lie in $\stackrel{\circ}{\Delta}_1 $ such that
\begin{equation} \label{A1} \frac{\mathcal{A}_{{\bf n},1}}{w_{{\bf n},1}}  \in \mathcal{H}(\mathbb{C} \setminus  \Delta_2) \qquad \mbox{and} \qquad \frac{\mathcal{A}_{{\bf n},1}}{w_{{\bf n},1}} = \mathcal{O}\left(\frac{1}{z^{|{\bf n}| - \overline{n}_1}}\right), \qquad z\to \infty,
\end{equation}
where $\overline{n}_j = \max \{n_k: k=j,\ldots,m\}$.

From  \eqref{A1} and Lemma \ref{reduc} it follows that
\begin{equation} \label{ortog2}
\int x^{\nu} \mathcal{A}_{{\bf n},2}(x) \frac{d\sigma_2(x)}{w_{{\bf n},1}(x)} = 0, \qquad \nu = 0,\ldots,|{\bf n}| - \overline{n}_1 -2,
\end{equation}
and
\begin{equation} \label{resto2}
\frac{\mathcal{A}_{{\bf n},1}(z)}{w_{{\bf n},1}(z)} = \int \frac{ \mathcal{A}_{{\bf n},2}(x) d\sigma_2(x)}{w_{{\bf n},1 }(x)(z-x)}.
\end{equation}
In particullar,  \eqref{ortog2} implies that $\mathcal{A}_{{\bf n},2}$ has at least $|{\bf n}| - \overline{n}_1 -1$ sign changes in $\stackrel{\circ}{\Delta}_2 $. (We cannot claim that $\mathcal{A}_{{\bf n},2}$ has exactly $|{\bf n}| - \overline{n}_1 -1$ simple zeros in $\mathbb{C} \setminus \Delta_3$ and that they all lie in $\stackrel{\circ}{\Delta}_2$ except if $\overline{n}_1 = n_1$.) Therefore, there exists a polynomial  $w_{{\bf n},2}, \deg w_{{\bf n},2} = |{\bf n}| - \overline{n}_1 -1$, whose zeros lie in $\stackrel{\circ}{\Delta}_2$, such that
\[ \frac{\mathcal{A}_{{\bf n},2}}{w_{{\bf n},2}}  \in \mathcal{H}(\mathbb{C} \setminus \Delta_3) \qquad \mbox{and} \qquad \frac{\mathcal{A}_{{\bf n},2}}{w_{{\bf n},2}} = \mathcal{O}\left(\frac{1}{z^{|{\bf n}| - \overline{n}_1 - \overline{n}_2}}\right) , \quad z \to \infty.
\]

Iterating this process, using Lemma \ref{reduc} several times, on step $j,\, j \in \{1,\ldots,m\},$ we find that there exists a polynomial
$w_{{\bf n},j}, \deg w_{{\bf n},j} = |{\bf n}| -\overline{n}_1 -\cdots - \overline{n}_{j-1}- 1,$ whose zeros are points where $\mathcal{A}_{{\bf n},j}$ changes sign in $\stackrel{\circ}{\Delta}_j $  such that
\begin{equation} \label{Anj} \frac{\mathcal{A}_{{\bf n},j}}{w_{{\bf n},j}}  \in \mathcal{H}(\mathbb{C} \setminus \Delta_{j+1}) \qquad \mbox{and} \qquad \frac{\mathcal{A}_{{\bf n},j}}{w_{{\bf n},j}} = \mathcal{O}\left(\frac{1}{z^{|{\bf n}| - \overline{n}_1 - \cdots-\overline{n}_j}}\right) , \quad z \to \infty.
\end{equation}
This process concludes as soon as $|{\bf n}| - \overline{n}_1 - \cdots-\overline{n}_j \leq 0$. Since $\lim_{{\bf n} \in \Lambda} |{\bf n}| = \infty$, because of \eqref{cond1} we can always take $m$ steps for all ${\bf n} \in \Lambda$ with $|{\bf n}|$ sufficiently large. In what follows, we only consider such ${\bf n}$'s.

When $n_1 = \overline{n}_1\geq \cdots \geq n_m = \overline{n}_m$, we obtain that $\mathcal{A}_{{\bf n},m} \equiv a_{{\bf n},m}$ has $n_m -1$ sign changes in $\stackrel{\circ}{\Delta}_m $ and since $\deg a_{{\bf n},m} \leq n_m-1$ this means that $\deg a_{{\bf n},m} = n_m-1$ and all its zeros lie in $\stackrel{\circ}{\Delta}_m $. (In fact, in this case we can prove that $\mathcal{A}_{{\bf n},j}, j=1,\ldots,m$ has exactly $|{\bf n}| - n_1-\cdots -n_{j-1} -1$ zeros in $\mathbb{C} \setminus \Delta_{j+1}$ that they are all simple and lie in  $\stackrel{\circ}{\Delta}_j $, where $\Delta_{m+1} = \emptyset$, compare with \cite[Propositions 2.5, 2.7]{FLLS}.)

In general, we have that $a_{{\bf n},m}$ has at least $|{\bf n}| - \overline{n}_1 - \cdots-\overline{n}_{m-1}-1$ sign changes in $\stackrel{\circ}{\Delta}_m $; therefore, the number of zeros of $a_{{\bf n},m}$ which may lie outside of $\Delta_m$ is bounded by
\[ \deg a_{{\bf n},m} - (|{\bf n}| - \overline{n}_1 - \cdots-\overline{n}_{m-1}-1) \leq \sum_{k=1}^{m-1} \overline{n}_k -n_k \leq (m-1)C,\]
where $C$ is the constant given in \eqref{cond1}, which does not depend on ${\bf n} \in \Lambda$.

For $j=m-1$ we have that there exists $w_{{\bf n},m-1}, \deg w_{{\bf n},m-1} = |{\bf n}| - \overline{n}_{1} - \cdots - \overline{n}_{m-2} -1,$ whose zeros lie in $\stackrel{\circ}{\Delta}_{m-1}$ such that
\[ \frac{\mathcal{A}_{{\bf n},m-1}}{w_{{\bf n},m-1}} = \frac{a_{{\bf n},m-1} + a_{{\bf n},m} \widehat{\sigma}_m}{w_{{\bf n},m-1}} \in \mathcal{H}(\mathbb{C} \setminus \Delta_m)\quad \mbox{and} \quad  \frac{\mathcal{A}_{{\bf n},m-1}}{w_{{\bf n},m-1}} = \mathcal{O}\left(\frac{1}{z^{|{\bf n}| - \overline{n}_1 - \cdots-\overline{n}_{m-1}}}\right) ,\,\,  z \to \infty,
\]
where $\deg  a_{{\bf n},m-1} \leq n_{m-1} -1, \deg  a_{{\bf n},m} \leq n_{m} -1$. Thus, using \eqref{cond1} it is easy to check that  $(a_{{\bf n},m-1}/a_{{\bf n},m})_{ n \in \Lambda}$ forms a sequence of  incomplete diagonal multi-point Pad\'e approximants of $-\widehat{\sigma}_m$ satisfying a)-b) with appropriate values of $n,\kappa$ and $\ell$. Due to Lemma \ref{BusLop} it follows that
\[ h-\lim_{n\in \Lambda} \frac{a_{{\bf n},m-1}}{a_{{\bf n},m}} = - \widehat{\sigma}_m, \qquad \mbox{inside} \qquad \mathbb{C} \setminus \Delta_m.
\]
Dividing by $\widehat{\sigma}_m$ and using \eqref{s22}, we also have
\[ \frac{\mathcal{A}_{{\bf n},m-1}}{\widehat{\sigma}_m w_{{\bf n},m-1}} = \frac{a_{{\bf n},m-1} \widehat {\tau}_m + b_{{\bf n},m-1}}{w_{{\bf n},m-1}} \in \mathcal{H}(\mathbb{C} \setminus \Delta_m),
\]
where $b_{{\bf n},m-1} = a_{{\bf n},m} + \ell_m a_{{\bf n},m-1}$ and
\[  \frac{\mathcal{A}_{{\bf n},m-1}}{\widehat{\sigma}_m w_{{\bf n},m-1}} = \mathcal{O}\left(\frac{1}{z^{|{\bf n}| - \overline{n}_1 - \cdots-\overline{n}_{m-1}-1}}\right) ,\,\,  z \to \infty.
\]
Consequently, $(b_{{\bf n},m-1}/a_{{\bf n},m-1})_{ n \in \Lambda}$ forms a sequence of  incomplete diagonal multi-point Pad\'e approximants of $-\widehat{\tau}_m$ satisfying a)-b) with appropriate values of $n,\kappa$ and $\ell$. Then Lemma  \ref{BusLop} and Corollary  \ref{momentosNik} imply that
\[ h-\lim_{n\in \Lambda} \frac{b_{{\bf n},m-1}}{a_{{\bf n},m-1}} = - \widehat{\tau}_m, \qquad \mbox{inside} \qquad \mathbb{C} \setminus \Delta_m,
\]
or, equivalently,
\[ h-\lim_{n\in \Lambda} \frac{a_{{\bf n},m}}{a_{{\bf n},m-1}} = - \widehat{\sigma}_m^{-1}, \qquad \mbox{inside} \qquad \mathbb{C} \setminus \Delta_m.
\]
We have proved \eqref{convHaus} for $j=m-1$.

For $j= m-2$, we have shown that there exists a polynomial
$w_{{\bf n},m-2}, \deg w_{{\bf n},m-2} = |{\bf n}| -\overline{n}_1 - \cdots \overline{n}_{m-3}- 1,$ whose zeros lie in $\stackrel{\circ}{\Delta}_{m-2} $ such that
\[ \frac{\mathcal{A}_{{\bf n},m-2}}{w_{{\bf n},m-2}} =  \frac{a_{{\bf n},m-2} + a_{{\bf n},m-1} \widehat{\sigma}_{m-1} + a_{{\bf n},m} \langle \sigma_{m-1},\sigma_m\widehat{\rangle}}{w_{{\bf n},m-2}} \in \mathcal{H}(\mathbb{C} \setminus \Delta_{m-1})
\]
and
\[  \frac{\mathcal{A}_{{\bf n},m-2}}{w_{{\bf n},m-2}} = \mathcal{O}\left(\frac{1}{z^{|{\bf n}| - \overline{n}_1 - \cdots-\overline{n}_{m-2}}}\right), \qquad z \to \infty.
\]
However, using \eqref{s22} and \eqref{4.4}, we obtain
\[\frac{a_{{\bf n},m-2} + a_{{\bf n},m-1} \widehat{\sigma}_{m-1} + a_{{\bf n},m} \langle \sigma_{m-1},\sigma_m\widehat{\rangle}}{\widehat{\sigma}_{m-1}} =
\]
\[
(\ell_{m-1} a_{{\bf n},m-2}+ a_{{\bf n},m-1} + C_1 a_{{\bf n},m})+ a_{{\bf n},m-2}\widehat{\tau}_{m-1}     - a_{{\bf n},m} \langle  {\tau}_{m-1}, \langle \sigma_{m}, \sigma_{m-1} \rangle \widehat{\rangle},
\]
where $\deg \ell_{m-1} = 1$ and $C_1$ is a constant. Consequently, $ {\mathcal{A}_{{\bf n},m-2}}/{\widehat{\sigma}_{m-1} } $ adopts the form of $\mathcal{A}$ in Lemma \ref{reduc}, $ {\mathcal{A}_{{\bf n},m-2}}/({\widehat{\sigma}_{m-1}w_{{\bf n},m-2} }) \in \mathcal{H}(\mathbb{C} \setminus \Delta_{m-1})$,   and
\begin{equation} \label{Anm} \frac{\mathcal{A}_{{\bf n},m-2}}{\widehat{\sigma}_{m-1}w_{{\bf n},m-2}} = \mathcal{O}\left(\frac{1}{z^{|{\bf n}| - \overline{n}_1 - \cdots-\overline{n}_{m-2} -1}}\right), \qquad z \to \infty.
\end{equation}
From \eqref{eq:4} in Lemma \ref{reduc} it follows that for $\nu = 0,\ldots, |{\bf n}| - \overline{n}_1 - \cdots-\overline{n}_{m-2} -3$
\[ \int x^{\nu}  \left({a_{{\bf n},m-2}(x)     - a_{{\bf n},m}(x)   \langle \sigma_{m}, \sigma_{m-1} \widehat{\rangle}(x)} \right) \frac{d  \tau_{m-1}(x)}{w_{{\bf n},m-2}(x)} = 0.
\]

Therefore, $ a_{{\bf n},m-2}- a_{{\bf n},m} \langle \sigma_{m}, \sigma_{m-1} \widehat{\rangle} \in \mathcal{H}(\mathbb{C} \setminus \Delta_m)$ must have at least $|{\bf n}| - \overline{n}_1 - \cdots-\overline{n}_{m-2} -2$ sign changes in $\stackrel{\circ}{\Delta}_{m-1}$. This means that there exists a polynomial $w_{{\bf n},m-2}^*, \deg w_{{\bf n},m-2}^* = |{\bf n}| - \overline{n}_1 - \cdots-\overline{n}_{m-2} -2,$ whose zeros are simple and lie in $\stackrel{\circ}{\Delta}_{m-1}$ such that
\[ \frac{a_{{\bf n},m-2}- a_{{\bf n},m} \langle \sigma_{m}, \sigma_{m-1} \widehat{\rangle}}{w_{{\bf n},m-2}^*} \in \mathcal{H}(\mathbb{C} \setminus \Delta_m)
\]
and
\[ \frac{a_{{\bf n},m-2}- a_{{\bf n},m} \langle \sigma_{m}, \sigma_{m-1} \widehat{\rangle}}{w_{{\bf n},m-2}^*}  = \mathcal{O}\left(\frac{1}{z^{{|{\bf n}| - \overline{n}_1 - \cdots-\overline{n}_{m-3}-2\overline{n}_{m-2}-1}}}\right), \qquad  z\to \infty.
\]
Due to \eqref{cond1}, this implies that $(a_{{\bf n},m-2}/a_{{\bf n},m}), n \in \Lambda,$ is a sequence of incomplete diagonal Pad\'e approximants of $\langle \sigma_{m}, \sigma_{m-1} \widehat{\rangle}$. By Lemma  \ref{BusLop} and Corollary \ref{momentosNik} we obtain its convergence in Hausdorff content to $\langle \sigma_{m}, \sigma_{m-1} \widehat{\rangle}$. To prove the other part in \eqref{convHaus}, we divide by $\langle \sigma_m,\sigma_{m-1}\widehat{\rangle}(z)$ use \eqref{s22} and proceed as we did in the case $j=m-1$.

Let us  prove \eqref{convHaus} in general. Fix $j \in \{0,\ldots,m-3\}$ (for $j=m-2,m-1$ it's been proved). Having in mind \eqref{Anj} we need to reduce $\mathcal{A}_{{\bf n},j}$ so as to eliminate all $a_{{\bf n},k}, k=j+1,\ldots,m-1$. We start out eliminating $a_{{\bf n},j+1}$. Consider the ratio $\mathcal{A}_{{\bf n},j}/\widehat{\sigma}_{j+1}$. Using \eqref{s22} and \eqref{4.4} we obtain that
\[ \frac{\mathcal{A}_{{\bf n},j}}{\widehat{\sigma}_{j+1}} =    \left(\ell_{j+1} a_{{\bf n},j}+ \sum_{k=j+1}^m \frac{|s_{j+1,k}|}{|\sigma_{j+1}|} a_{{\bf n},j+1} \right) + a_{{\bf n},j}\widehat{\tau}_{j+1} - \sum_{k=j+2}^m a_{{\bf n},k} \langle  {\tau}_{j+1}, \langle s_{j+2,k}, \sigma_{j+1} \rangle \widehat{\rangle},
\]
has the form of $\mathcal{A}$ in Lemma \ref{reduc}, where $ {\mathcal{A}_{{\bf n},j}}/({\widehat{\sigma}_{j+1}w_{{\bf n},j}}) \in \mathcal{H}(\mathbb{C} \setminus \Delta_{j+1})$, and
\[\frac{\mathcal{A}_{{\bf n},j}}{\widehat{\sigma}_{j+1}w_{{\bf n},j}} \in \mathcal{O}\left(\frac{1}{z^{|{\bf n}| - \overline{n}_1 - \cdots-\overline{n}_{j} -1}}\right), \qquad z \to \infty.
\]
From \eqref{eq:4} of Lemma \ref{reduc}, we obtain that for $\nu = 0,\ldots,|{\bf n}| - \overline{n}_1 - \cdots-\overline{n}_{j} -3$
\[ 0 = \int x^{\nu} \left( a_{{\bf n},j}(x) - \sum_{k=j+2}^m a_{{\bf n},k}  \langle s_{j+2,k}, \sigma_{j+1}  \widehat{\rangle}(x) \right)\frac{d\tau_{j+1}(x)}{w_{{\bf n},j}(x)}
\]
which implies that the function in parenthesis under the integral sign has at least $|{\bf n}| - \overline{n}_1 - \cdots-\overline{n}_{j} -2$ sign changes in $\stackrel{\circ}{\Delta}_{j+1} $. In turn, it follows that there exists  a polynomial $\widetilde{w}_{{\bf n},j }, \deg \widetilde{w}_{{\bf n},j } = |{\bf n}| - \overline{n}_1 - \cdots-\overline{n}_{j} -2$, whose zeros are simple and lie in $\stackrel{\circ}{\Delta}_{j+1} $ such that
\[ \frac{a_{{\bf n},j}  - \sum_{k=j+2}^m a_{{\bf n},k}  \langle s_{j+2,k}, \sigma_{j+1}  \widehat{\rangle} }{\widetilde{w}_{{\bf n},j }} \in \mathcal{H}(\mathbb{C} \setminus \Delta_{j+2})
\]
and
\[ \frac{a_{{\bf n},j}  - \sum_{k=j+2}^m a_{{\bf n},k}  \langle s_{j+2,k}, \sigma_{j+1}  \widehat{\rangle} }{\widetilde{w}_{{\bf n},j }} = \mathcal{O}\left(\frac{1}{z^{|{\bf n}| - \overline{n}_1 - \cdots-\overline{n}_{j-1}-2\overline{n}_{j} -1}}\right), \qquad z \to \infty.
\]
Notice that $a_{{\bf n},j+1}$ has been eliminated and that
\[ \langle s_{j+2,k}, \sigma_{j+1}   {\rangle} = \langle \langle \sigma_{j+2},\sigma_{j+1}\rangle, \sigma_{j+3}, \ldots,\sigma_k \rangle, \qquad k = j+3,\ldots,m.
\]

Now we must do away with $a_{{\bf n},j+2}$ in  $a_{{\bf n},j}  - \sum_{k=j+2}^m a_{{\bf n},k}  \langle s_{j+2,k}, \sigma_{j+1}  \widehat{\rangle}$ (in case that $j+2 < m$). To this end, we consider the ratio
\[ \frac{a_{{\bf n},j}  - \sum_{k=j+2}^m a_{{\bf n},k}  \langle s_{j+2,k}, \sigma_{j+1}  \widehat{\rangle} }{\langle \sigma_{j+2},\sigma_{j+1}\widehat{\rangle}}
\]
and repeat the arguments employed above with $\mathcal{A}_{{\bf n},j}$. After $m-j-2$ reductions obtained applying consecutively Lemma \ref{reduc}, we find that there exists a polynomial which we denote $w_{{\bf n},j}^*, \deg w_{{\bf n},j}^* = |{\bf n}| - \overline{n}_1 - \cdots-\overline{n}_{j-1}-(m-j-1)\overline{n}_{j} -2$, whose zeros are simple and lie in $\stackrel{\circ}{\Delta}_{m-1} $ such that
\[ \frac{a_{{\bf n},j} - (-1)^{m-j} a_{{\bf n},m} \langle \sigma_{m},\ldots, \sigma_{j+1} \widehat{\rangle}}{w_{{\bf n},j}^*} \in \mathcal{H}(\mathbb{C} \setminus \Delta_m)
\]
and
\[ \frac{a_{{\bf n},j}- (-1)^{m-j} a_{{\bf n},m} \langle \sigma_{m},\ldots, \sigma_{j+1} \widehat{\rangle}}{w_{{\bf n},j}^*}  = \mathcal{O}\left(\frac{1}{z^{{|{\bf n}| - \overline{n}_1 - \cdots-\overline{n}_{j-1}-(m-j)\overline{n}_{j}-1}}}\right), \qquad z\to \infty.
\]
Dividing by $(-1)^{m-j}  \langle \sigma_{m},\ldots, \sigma_{j+1} \widehat{\rangle},$ from here we also get that
\[ \frac{a_{{\bf n},j}(-1)^{m-j}\langle \sigma_{m},\ldots, \sigma_{j+1} \widehat{\rangle}^{-1} -  a_{{\bf n},m} }{w_{{\bf n},j}^*} \in \mathcal{H}(\mathbb{C} \setminus \Delta_m)
\]
and
\[ \frac{a_{{\bf n},j}(-1)^{m-j}\langle \sigma_{m},\ldots, \sigma_{j+1} \widehat{\rangle}^{-1}- a_{{\bf n},m} }{w_{{\bf n},j}^*}  = \mathcal{O}\left(\frac{1}{z^{{|{\bf n}| - \overline{n}_1 - \cdots-\overline{n}_{j-1}-(m-j)\overline{n}_{j}-2}}}\right), \qquad z\to \infty.
\]
On account of \eqref{cond1}, these relations imply that $(a_{{\bf n},j}/a_{{\bf n},m}), {\bf n} \in \Lambda,$ is a sequence of incomplete diagonal multi-point Pad\'e approximants of
$(-1)^{m-j}\langle \sigma_{m},\ldots, \sigma_{j+1} \widehat{\rangle}$ and $(a_{{\bf n},m}/a_{{\bf n},j}), {\bf n} \in \Lambda,$ is a sequence of incomplete diagonal multi-point Pad\'e approximants of $(-1)^{m-j}\langle \sigma_{m},\ldots, \sigma_{j+1} \widehat{\rangle}^{-1}$. Since $\langle \sigma_{m},\ldots, \sigma_{j+1} \widehat{\rangle}^{-1} = \widehat{\tau}_{m,j+1} + \ell_{m,j+1},$ from Lemma   \ref{BusLop} and Corollary \ref{momentosNik} we obtain \eqref{convHaus}.

Going one step further using Lemma \ref{reduc}, we also obtain that
\[ 0 = \int x^{\nu} a_{{\bf n},j}(x)  \frac{d \tau_{m,j+1}}{w_{{\bf n},j}^*(x)},\qquad \nu=0,\ldots,|{\bf n}| - \overline{n}_1 - \cdots-\overline{n}_{j-1}-(m-j)\overline{n}_{j}-4
\]
which implies that $a_{{\bf n},j}$ has at least $|{\bf n}| - \overline{n}_1 - \cdots-\overline{n}_{j-1}-(m-j)\overline{n}_{j}-3$ sign changes in $\stackrel{\circ}{\Delta}_{m} $. From \eqref{cond1} we obtain that there exists a non-negative constant $C_1$, independent of $n\in \Lambda$, such that the number of zeros of $a_{{\bf n},j}, j=0,\ldots,m,$ in $\Delta_m$ is bounded from below by $(|{\bf n}|/m) - C_1$. This settles the last statement.
\end{proof}

In the case of decreasing components in $\bf n$, we saw that all the zeros of $a_{{\bf n},m}$ lie in $\Delta_m$ and \cite[Lemma 1]{gon} would allow us to derive immediately uniform convergence on each compact subset of $\mathbb{C} \setminus \Delta_m$ from the convergence in Hausdorff content. For other configurations of the components we have to work a little harder.

\begin{proof}[Proof of Theorem \ref{converge}] Let $\overline{\jmath}$ be the last component of $(n_0,\ldots,n_m)$ such that $n_{\overline{\jmath}} = \min_{j=0,\ldots,m} (n_j)$.  Let us prove that $\deg a_{{\bf n},\overline{\jmath}} = n_{\overline{\jmath}} -1$, that all its zeros are simple and lie in $\stackrel{\circ}{\Delta}_{m} $.

From  \cite[Theorem 3.2]{FL4} (see also \cite[Theorem 1.3]{FL3}) we know that  there exists a permutation $\lambda$ of $(0,\ldots,m)$ which reorders the components of $(n_0,n_1,\ldots,n_m)$ decreasingly, $n_{\lambda(0)} \geq \cdots \geq n_{\lambda(m)},$ and an associated Nikishin system $(r_{1,1},\ldots,r_{1,m}) = {\mathcal{N}}(\rho_{1},\ldots,\rho_m)$ such that
\[  \mathcal{A}_{{\bf n},0} = (q_{{\bf n},0} + \sum_{k=1}^m q_{{\bf n},k} \widehat{r}_{1,k})\widehat{s}_{1,\lambda(0)}, \qquad \deg q_{{\bf n},k} \leq n_{\lambda(k)} -1, \qquad k=0,\ldots,m.
\]
The permutation may be taken so that for all $0 \leq j <k \leq n$ with $n_j = n_k$ then also $\lambda(j) < \lambda(k)$. In this case,  see formulas (31) in the proof of \cite[Lemma 2.3]{FL3}, it follows that $q_{{\bf n},m} = \pm a_{{\bf n},\overline{\jmath}}$. Reasoning with $q_{{\bf n},0} + \sum_{k=1}^m q_{{\bf n},k} \widehat{r}_{1,k}$  as we did with $\mathcal{A}_{{\bf n},0}$ we obtain that
$\deg q_{{\bf n},m} = n_{\lambda(m)} -1$ and that its zeros are all simple and lie in $\stackrel{\circ}{\Delta}_{m} $. However, $n_{\lambda(m)} = n_{\overline{\jmath}}$ and $q_{{\bf n},m} = \pm a_{{\bf n},\overline{\jmath}}$ so the statement holds.

The index $\overline{\jmath}$ as defined above may depend on the multi-index ${\bf n} \in \Lambda$. Given $\overline{\jmath} \in \{0,\ldots,m\}$, let us denote by $\Lambda(\overline{\jmath})$ the set of all ${\bf n} \in \Lambda$ such that $\overline{\jmath}$ is the last component of $(n_0,\ldots,n_m)$ such that $n_{\overline{\jmath}} = \min_{j=0,\ldots,m} (n_j)$. Fix $\overline{\jmath}$ and suppose that $\Lambda(\overline{\jmath})$ contains infinitely many multi-indices. If $\overline{\jmath} = m$, then  \cite[Lemma 1]{gon} and the first limit in \eqref{convHaus} imply  that
\[ \lim_{n\in \Lambda(m)}\frac{a_{{\bf n}, j}}{a_{{\bf n},m}} = (-1)^{m-j}\widehat{s}_{m,j+1}, \qquad j=0,\ldots,m-1,
\]
uniformly on each compact subset of $\mathbb{C}\setminus \Delta_m$, as needed.

Assume that $\overline{\jmath} \in \{0,\ldots,m-1\}$. Since all the zeros of $a_{{\bf n},\overline{\jmath}}$   lie in $\stackrel{\circ}{\Delta}_m,$ using \cite[Lemma 1]{gon} and the second limit in \eqref{convHaus} for $j = \overline{\jmath}$, we obtain that
\begin{equation} \label{convinv} \lim_{{\bf n} \in \Lambda(\overline{\jmath})} \frac{a_{{\bf n},m}}{a_{{\bf n},\overline{\jmath}}} = \frac{1}{(-1)^{m-\overline{\j}}\widehat{s}_{m,\overline{\jmath} +1}},
\end{equation}
uniformly on each compact subset of $\mathbb{C}\setminus \Delta_m$.  The function on the right hand side of \eqref{convinv} is holomorphic and never zero on $\mathbb{C} \setminus \Delta_m$ and the approximating functions are holomorphic on $\mathbb{C} \setminus \Delta_m$. Using Rouche's theorem it readily follows  that on any compact subset  $\mathcal{K} \subset \mathbb{C} \setminus \Delta_m$ for all sufficiently large $|{\bf n}|, n \in \Lambda(\overline{\jmath}),$ the polynomials $a_{{\bf n},m} $ have no zero on $\mathcal{K}$. This is true for any $\overline{\jmath} \in \{0,\ldots,m\}$ such that $\Lambda(\overline{\jmath})$ contains infinitely many multi-indices. Therefore, the only accumulation points of the zeros of the polynomials ${a_{{\bf n},m}}$ are   in $\Delta_m \cup \{\infty\}$.

Hence, on any  bounded region $D$ such that $\overline{D} \subset \mathbb{C} \setminus \Delta_m$ for each fixed $j=0,\ldots,m-1,$ and all sufficiently large $|{\bf n}|, {\bf n} \in \Lambda$, we have that $ {a_{{\bf n},j}}/{a_{{\bf n},m}} \in \mathcal{H}(D)$. From \cite[Lemma 1]{gon} and the first part of \eqref{convHaus} it follows that
\begin{equation} \label{convunifgen} \lim_{{\bf n} \in \Lambda} \frac{a_{{\bf n},j}}{a_{{\bf n},m}} = (-1)^{m-j}\widehat{s}_{m,j+1}, \qquad j=0,\ldots,m-1,
\end{equation}
uniformly on each compact subset of $D$. Since $D$ was chosen arbitrarily, as long as $\overline{D} \subset \mathbb{C} \setminus \Delta_m$, it follows that the convergence is uniform on each compact subset of $\mathbb{C} \setminus \Delta_m$ and we have \eqref{fund1}. Since the right hand of \eqref{fund1} is a function which does not vanish in $\overline{D} \subset \mathbb{C} \setminus \Delta_m$, Rouche's theorem implies that for each $j=0,\ldots,m-1$ the accumulation points of the zeros of the polynomials $a_{{\bf n},j}$ must be in $\Delta_m \cup \{\infty\}$ as claimed. (For $j=m$ this was proved above.)

Now,
\[ \frac{\mathcal{A}_{{\bf n},j}}{a_{{\bf n},m}} = \frac{a_{{\bf n},j}}{a_{{\bf n},m}} + \sum_{k=j+1}^{m-1} \frac{a_{{\bf n},k}}{a_{{\bf n},m}} \widehat{s}_{j+1,k} + \widehat{s}_{j+1,m}.
\]
According to formula (17) in \cite[Lemma 2.9]{FL4}
\[ 0 \equiv (-1)^{m-j }\widehat{s}_{m,j+1}  + \sum_{k=j+1}^{m-1}(-1)^{m-k}   \widehat{s}_{m,k+1} \widehat{s}_{j+1,k} +
  \widehat{s}_{j+1, m}, \quad z \in \mathbb{C} \setminus (\Delta_{j+1} \cup \Delta_m).
\]
Deleting one expression from the other we have that
\begin{equation} \label{difer} \frac{\mathcal{A}_{{\bf n},j}}{a_{{\bf n},m}} = \left(\frac{a_{{\bf n},j}}{a_{{\bf n},m}} - (-1)^{m-j }\widehat{s}_{m,j+1}\right) + \sum_{k=j+1}^{m-1} \left(\frac{a_{{\bf n},k}}{a_{{\bf n},m}} - (-1)^{m-k}   \widehat{s}_{m,k+1} \right) \widehat{s}_{j+1,k}
\end{equation}
Consequently, for each $j=0,\ldots,m-1$, from \eqref{fund1} we obtain \eqref{fund2}.
\end{proof}

Suppose that $\Delta_m$ is bounded. Let $\Gamma$ be a positively oriented closed simple Jordan curve that surrounds $\Delta_m$. Define $\kappa_{{\bf n},j}(\Gamma), j=0,\ldots,m$ to be  the number of zeros of ${a_{{\bf n},j}}$ outside $\Gamma$. As above, given $\overline{\jmath} \in \{0,\ldots,m\}$, let us denote by $\Lambda(\overline{\jmath})$ the set of all ${\bf n} \in \Lambda$ such that $\overline{\jmath}$ is the last component of $(n_0,\ldots,n_m)$ which satisfies $n_{\overline{\jmath}} = \min_{j=0,\ldots,m} (n_j)$.

\begin{corollary} \label{ceros} Suppose that the assumptions of Theorem \ref{converge} hold and $\Delta_m$ is bounded. Then
for all sufficiently large $|{\bf n}|, {\bf n} \in \Lambda(\overline{\jmath}),$
\begin{equation} \label{kappaj} \kappa_{{\bf n},j}(\Gamma) = \left\{
\begin{array}{ll}
n_j - n_{\overline{\jmath}}\,\,, & j=0,\ldots,m-1, \\
n_m - n_{\overline{\jmath}} -1, & j=m.
\end{array}
\right.
\end{equation}
The rest of the zeros of the polynomials $a_{{\bf n},j}$ accumulate (or lie) on $ {\Delta_{m}} $.
\end{corollary}

\begin{proof}    Fix $\overline{\jmath} \in \{0,\ldots,m-1\}$. Assume that $\Lambda(\overline{\jmath})$ contains infinitely many multi-indices. Using the argument principle and \eqref{convinv} it follows that
\[ \lim_{{\bf n}\in \Lambda(\overline{\jmath})} \frac{1}{2\pi i} \int_{\Gamma} \frac{({a_{{\bf n},m}}/{a_{{\bf n},\overline{\jmath}}})^{\prime}(z)}{({a_{{\bf n},m}}/{a_{{\bf n},\overline{\jmath}}})(z)} d z =  \frac{1}{2\pi i} \int_{\Gamma} \frac{(1/\widehat{s}_{m,\overline{\jmath} +1})^{\prime}(z)}{(1/\widehat{s}_{m,\overline{\jmath} +1})(z)} d z = 1,
\]
because $1/\widehat{s}_{m,\overline{\jmath} +1}$ has one pole and no zeros outside $\Gamma$ (counting the point $\infty$).   Recall that $\deg a_{{\bf n},j} = n_j -1, j=0,\ldots,m$ and that all the zeros of $a_{{\bf n},\overline{\jmath}}$ lie on $\Delta_m$. Then,  for all sufficiently large $|{\bf n}|, {\bf n} \in \Lambda(\overline{\jmath}),$
\[ (n_m -1) - (n_{\overline{\jmath}} -1) -  \kappa_{{\bf n},m}(\Gamma) = 1.
\]
Consequently,
\begin{equation} \label{kappam} \kappa_{{\bf n},m}(\Gamma) = n_m - n_{\overline{\jmath}} -1, \qquad {\bf n} \in \Lambda(\overline{\jmath}).
\end{equation}
Analogously, from \eqref{convunifgen}, for $j=0,\ldots,m-1$, we obtain
\[ \lim_{{\bf n}\in \Lambda} \frac{1}{2\pi i} \int_{\Gamma} \frac{({a_{{\bf n},j}}/{a_{{\bf n}, m}})^{\prime}(z)}{({a_{{\bf n},j}}/{a_{{\bf n},m}})(z)} d z =  \frac{1}{2\pi i} \int_{\Gamma} \frac{ \widehat{s}_{m, {j} +1}^{\prime}(z)}{ \widehat{s}_{m, {j} +1} (z)} d z = -1.
\]
Therefore, for all sufficiently large $|{\bf n}|, {\bf n} \in \Lambda,$
\[ n_j - n_m  + \kappa_{{\bf n},m}(\Gamma) - \kappa_{{\bf n},j}(\Gamma) = -1, \qquad j=0,\ldots,m-1,
\]
which together with \eqref{kappam} gives \eqref{kappaj}.
The last statement follows from the fact that the only accumulation points of the zeros of the $a_{{\bf n},j}$ are in $\Delta_m \cup \{\infty\}$.
\end{proof}

\begin{remark} The thesis of Theorem \ref{converge} remains valid if in place of  \eqref{cond1} we require that
\begin{equation} \label{cond2} n_j  =  \frac{|{\bf n}|}{m}  + o(|{\bf n}|),\qquad   |{\bf n}| \to \infty ,  \qquad j=1,\ldots,m.
\end{equation}
To prove this we need an improved version of Lemma \ref{BusLop} in which  the parameter $\ell$ in b)  depends on $n$ but $\ell(n) = o(n), n \to \infty$. The proof of Lemma 2 in \cite{BL} admits this variation with some additional technical difficulties in part resolved in the proof of \cite[Corollary 1]{FL2}.
\end{remark}

\begin{remark} If either $\Delta_m$ or $\Delta_{m-1}$ is a compact set and $\Delta_{m-1} \cap \Delta_{m} = \emptyset$, it not difficult to show that convergence takes place in \eqref{fund1} and \eqref{fund2} with geometric rate. More precisely, for $j=0,\ldots,m-1,$ and $\mathcal{K} \subset \mathbb{C} \setminus \Delta_m$, we have
\begin{equation} \label{asin1} \limsup_{{\bf n} \in \Lambda} \left\|\frac{a_{{\bf n}, j}}{a_{{\bf n},m}} - (-1)^{m-j}\widehat{s}_{m,j+1} \right\|_{\mathcal{K}}^{1/|{\bf n}|} = \delta_j < 1.
\end{equation}
For $j=0,\ldots,m-1,$ and $\mathcal{K} \subset \mathbb{C} \setminus (\Delta_{j+1} \cup \Delta_m)$
\begin{equation} \label{asin2} \limsup_{{\bf n} \in \Lambda} \left\|\frac{\mathcal{A}_{{\bf n},j}}{a_{{\bf n},m}}\right\|_{\mathcal{K}}^{1/|{\bf n}|} \leq \max\{\delta_k:j \leq k \leq m-1\} < 1.
\end{equation}
The second relation trivially follows from the first and $\eqref{difer}$.
The proof of the first is similar to that of \cite[Corollary 1]{FL2}. It is based on the fact that the number of interpolation points on $\Delta_{m-1}$ is $\mathcal{O}(|{\bf n}|), |{\bf n}| \to \infty,$ and that the distance from $\Delta_m$ to $\Delta_{m-1}$ is positive. Relations \eqref{asin1} and \eqref{asin2} are also valid if \eqref{cond1} is replaced with \eqref{cond2}.

Asymptotically, \eqref{cond2} still means that the components of $\bf n$ are equally valued. One can relax \eqref{cond2} requiring, for example, that the generating measures are regular in the sense of \cite[Chapter 3]{stto} in which case the exact asymptotics of \eqref{asin1} and \eqref{asin2} can be given (see \cite{Nik0}, \cite[Chapter 5, Section 7]{NS}, \cite[Theorem 5.1, Corollary 5.3]{FLLS}, and \cite[Theorem 1]{RS1}.
\end{remark}

\begin{remark}  The previous results can be applied to other approximation schemes. Let $S^1= \mathcal{N}(\sigma_0^1,\ldots,\,\sigma^1_{m_1}), S^2
=\mathcal{N}(\sigma^2_0,\ldots,\,\sigma^2_{m_2}), \sigma_0^1 =
\sigma_0^2$ be given.
Fix ${\bf
n}_1=(n_{1,0},\,n_{1,1},\ldots,\,n_{1,m_1})\in{{\mathbb{Z}}}_+^{m_1+1}$
and ${\bf
n}_2=(n_{2,0},\,n_{2,1},\ldots,\,n_{2,m_2})\in{{\mathbb{Z}}}_+^{m_2+1},
|{\bf n}_2| = |{\bf n}_1| -1$. Let ${\bf n} = ({\bf n}_1,{\bf n}_2)$. There exists a non-zero vector polynomial with real coefficients
$(a_{{\bf
n},0},\ldots,a_{{\bf n},m_1}),$ $\deg (a_{{\bf n},k}) \leq
n_{1,k} -1,  k=0,\ldots,m_1,$
such that for $j=0,\ldots,m_2,$
\[
 \int x^{\nu} \mathcal{A}_{{\bf n},0}(x) d s^2_{j}(x)  =0, \qquad \nu =0,\ldots,n_{2,j}
 -1,
\]
where
\[ \mathcal{A}_{{\bf n},0} = a_{{\bf n},0}  + \sum_{k=1}^{m_1} a_{{\bf n},k} \widehat{s}^1_{1,k}.
\]
In other words
\begin{equation}\label{definition**}
 \int \left(b_{{\bf n},0}(x) + \sum_{j=1}^{m_2} b_{{\bf n},j}(x) \widehat{s}_{1,j}^2(x)\right) \mathcal{A}_{{\bf n},0}(x) d \sigma^2_{0}(x)  =0, \qquad \deg b_{{\bf n},j} \leq n_{2,j} -1.
\end{equation}
This implies that $\mathcal{A}_{{\bf n},0}$ has exactly $|\bf n_2|$ zeros in $\mathbb{C} \setminus \Delta_1^1$ they are all simple and lie in $\stackrel{\circ}{\Delta^1 _0}$ (see \cite[Theorem 1.2]{FL4}. Here $\Delta_0^1 = \mbox{Co}(\supp(\sigma_0^1))$ and $\Delta_1^1 = \mbox{Co}(\supp(\sigma_1^1))$.  Therefore, $\left(a_{{\bf n},0}, \ldots, a_{{\bf n},m}\right)$ is a  type I multi-point Hermite-Pad\'e approximation of $(\widehat{s}_{1,1},\ldots,\widehat{s}_{1,m})$ with respect to $w_{\bf n}$ and the  results of this paper may be applied.

\end{remark}

\end{document}